\documentclass[reqno]{amsart}     
\usepackage{amsmath}
\usepackage{amsthm}
\usepackage{mathrsfs}
 
\usepackage{a4wide} 
\usepackage[T1]{fontenc}
\usepackage[latin1]{inputenc}
\usepackage{marginnote}
\usepackage{tikz}

\usepackage{xcolor}
\usepackage{caption}
\usepackage{hyperref}
\usepackage{array}
\usepackage{graphicx}

\renewcommand{\(}{\left(}
\renewcommand{\)}{\right)}
\newtheorem{thm}{Theorem}
\newtheorem{prop}{Proposition}
\newtheorem{lemma}{Lemma}
\newtheorem{cor}{Corollary\!\!}

\newtheorem{ncor}{Corollary}

\theoremstyle{definition}

\newtheorem{df}{Definition}
\newtheorem{ex}{Example}

\theoremstyle{remark}

\newtheorem{rem}{Remark\!\!}
 
\newtheorem{nrem}{Remark}

\newcommand{\bth}{\begin{thm}} 
\newcommand{\bl}{\begin{lemma}} 
\newcommand{\el}{\end{lemma}} 
\newcommand{\bp}{\begin{prop}} 
\newcommand{\ep}{\end{prop}} 
\newcommand{\bdf}{\begin{df}} 
\newcommand{\edf}{\end{df}} 
\newcommand{\brem}{\begin{rem}} 
\newcommand{\erem}{\end{rem}} 
\newcommand{\bnrem}{\begin{nrem}} 
\newcommand{\enrem}{\end{nrem}} 
\newcommand{\bex}{\begin{ex}} 
\newcommand{\eex}{\end{ex}} 
\newcommand{\bcor}{\begin{cor}} 
\newcommand{\ecor}{\end{cor}} 
\newcommand{\bncor}{\begin{ncor}} 
\newcommand{\encor}{\end{ncor}} 
\newcommand{\bpf}{\begin{proof}} 
\newcommand{\epf}{\end{proof}}

\begin{document}

\title{Protection numbers in simply generated trees and P\'olya trees} 
\author[B. Gittenberger, Z. Go{\l}\k{e}biewski, I. Larcher, M. Sulkowska]{Bernhard Gittenberger,
Zbigniew Go{\l}\k{e}biewski, \\ Isabella Larcher, \and Ma{\l}gorzata Sulkowska} 
\thanks{This research has been supported by the \"OAD, grant PL04-2018, and 
Wroclaw Univeristy of Science and Technology grant 0401/0052/18. The first and third author 
have also been supported by the Austrian Science Fund (FWF), grant SFB F50-03.} 
\address{Department of Discrete Mathematics and Geometry, Technische  
Universit\"at Wien, Wiedner Hauptstra\ss e 8-10/104, A-1040 Wien, Austria.}
\email{gittenberger@dmg.tuwien.ac.at} 
\address{
Department of Computer Science, Wroc{\l}aw University of Science and Technology, ul. Wybrzeze
Wyspianskiego 27, 50-370, Wroc\l aw, Poland.}
\email{zbigniew.golebiewski@pwr.edu.pl}
\address{Department of Discrete Mathematics and Geometry, Technische
Universit\"at Wien, Wiedner Hauptstra\ss e 8-10/104, A-1040 Wien, Austria.}
\email{isabella.larcher@tuwien.ac.at}
\address{
Department of Computer Science, Wroc{\l}aw University of Science and Technology, ul. Wybrzeze
Wyspianskiego 27, 50-370, Wroc\l aw, Poland.}
\email{malgorzata.sulkowska@pwr.edu.pl}


\date{\today}

\begin{abstract} 
We determine the limit of the expected value and the variance of the protection number of the
root in simply generated trees, in P\'olya trees, and in unlabelled non-plane binary trees, when
the number of vertices tends to infinity.  Moreover, we compute expectation and variance of the
protection number of a randomly chosen vertex in all those tree classes. We obtain exact formulas
as sum representations, where the obtained sums are rapidly converging and therefore allowing an
efficient numerical computation of high accuracy. Most proofs are based on a singularity analysis
of generating functions.
\end{abstract} 
 
\maketitle

\section{Introduction}

{\it The protection number of a tree} is the length of the shortest path from the root to a leaf.
It is interchangeably called {\it the protection number of a root}. We define {\it the protection
number of a vertex $v$} in tree $T$ as the protection number of a maximal subtree of $T$ having
$v$ as a root. We say that a vertex is {\it $k$-protected} if $k$ does not exceed its protection
number.

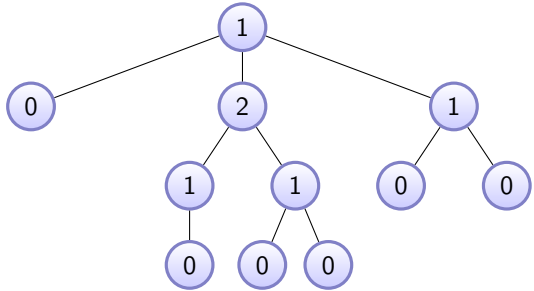
\begin{figure}[!ht]
\begin{center}
\begin{tikzpicture}
[
level 1/.style={sibling distance=8em, level distance=3em},
level 2/.style={sibling distance=4em, level distance=3em},
level 3/.style={sibling distance=2.5em, level distance=3em},
internal/.style={
	circle,
	draw=blue!55!black!50,
	top color=white,
	bottom color=blue!20,
	very thick,
	font=\sffamily
}
]

\node[internal]{1}
	child{ node[internal]{0} }
	child{ node[internal]{2} 
		child{ node[internal]{1} 
			child{ node[internal]{0} }		
		}
		child{ node[internal]{1} 
			child{ node[internal]{0} }
			child{ node[internal]{0} }		
		}
	}
	child{ node[internal]{1} 
		child{ node[internal]{0} }
		child{ node[internal]{0} }	
	};
\end{tikzpicture}
\end{center}
\caption{Tree with vertices holding their protection numbers.}   
\end{figure}

Previous research concerning protection numbers has been conducted in two closely related directions: (i) a number of $k$-protected vertices in a tree of size $n$, and (ii) the protection number of a root or a random vertex.

Cheon and Shapiro~\cite{DBLP:journals/appml/CheonS08} were the first ones to investigate the
number of $2$-protected nodes in trees. They stated the results for unlabelled ordered trees and
Motzkin trees. Later on Mansour~\cite{DBLP:journals/appml/Mansour11} complemented their work by
solving $k$-ary tree case. Over the next several years these results were followed by a series of
papers examining the number of $k$-protected nodes (usually for small values of $k$) in various
models of random trees. To mention just a few, Du and Prodinger~\cite{DBLP:journals/appml/DuP12}
analysed the average number of $2$-protected nodes in random digital search trees,  Mahmoud and
Ward~\cite{DBLP:journals/appml/MahmoudW12} presented a central limit theorem as well as exact
moments of all orders for the number of $2$-protected nodes in binary search trees and three years
later they found the number of $2$-protected nodes in recursive trees
(consult~\cite{DBLP:journals/jap/MahmoudW15}). The family of binary search trees was investigated
also by B\'ona and Pittel~\cite{bona_pittel_2017} who showed that the number of its $k$-protected
nodes decays exponentially in $k$.




In 2015 Holmgren and Janson~\cite{holmgren2015} went for more general results. Using probabilistic
methods, they derived a normal limit law for the number of $k$-protected nodes in a binary search
tree and a random recursive tree.

Soon after, two particular parameters attracted attention of the algorithmic community. These were
(as already mentioned earlier) the protection number of the root and the protection number of a random vertex.
In 2017 Copenhaver~\cite{Copenhaver:2017} found that in a random unlabelled plane tree the expected value of the protection number of the root and the expected value of the protection number of a random vertex approach $1.62297$ and $0.727649$, respectively, as the size of the tree tends to infinity. These results were extended by Heuberger and Prodinger~\cite{heuberger2017protection}. They showed the exact formulas for the first terms of the expectation, the variance and the probability of the respective protection numbers.

The protection number of a root is closely related to parameters called {\it minimal fill-up
level} and {\it saturation level}. These were studied previously by, among others,
Devroye~\cite{DBLP:journals/rsa/Devroye92} and Drmota~\cite{Drmota:2000, Drm09}.


The aim of this paper is to generalize the protection number results to a larger class of rooted
trees. We study both the root protection number as well as a random vertex protection number for
the family of simply generated trees (introduced by Meir and Moon~\cite{MM78}) and their non-plane
counterparts: unlabelled non-plane rooted trees, also called P\'olya trees due to their first
extensive treatment by P\'olya~\cite{Polya1937}, examined further by Otter~\cite{otter1948}
including numerical results and the binary case. The present paper broadens the results
from~\cite{heuberger2017protection}, but maintaining the emphasis on as concrete formulas as
possible. 

For simply generated trees a general theory of asymptotics of certain functional was developed
recently in \cite{DDS18}, but this theory does not cover local functionals as the number of
protected nodes. Devroye and Janson~\cite{devroye2014} presented a unified approach to obtaining
the number of $k$-protected nodes in various classes or random trees by putting them in the
general context of fringe subtrees introduced by Aldous in~\cite{Aldous91}. We have obtained
analogous results for simply generated trees, but employing a different methodology. This allows
an efficient numerical treatment and may serve as a basis for random generation in the framework
of Boltzmann sampling \cite{DFLS04}. Parts of our investigations fall into the general framework
of additive functionals treated in \cite{Wa15}, but our focus on concrete expressions allows an
easy access to numerical evaluation of the considered parameters.

\subsection*{Plan of the paper} In Sections~\ref{SGT},~\ref{PT}, and~\ref{NPBT} we consider simply
generated trees, P\'olya trees and non-plane binary trees, respectively. In each section the
expected value and the variance of the protection number of the root and the protection number
of a random vertex are computed. All these quantities tend to constants when the tree size tends
to infinity. The emphasis is on deriving exact expressions for these constants in terms of
characteristic parameters of the considered tree class. We obtain them in terms of sums that
converge at an exponential rate and therefore enable us to compute efficiently accurate numerical
values. We provide numerical values for the several well-known simply generated tree classes as
well as for the two non-plane classes studied in Sections~\ref{PT} and~\ref{NPBT}.

\section{Simply generated trees}\label{SGT}

\subsection{Protection number of the root}
The class $\mathcal{T}$ of simply generated trees was introduced in \cite{MM78} and can be described as the class of plane rooted trees whose generating function satisfies a function equation of particular type: If $t_n$ denotes the sum of the weights of all trees with $n$ vertices, then the generating function $T(z)=\sum_{n\ge 0} t_nz^n$ satisfies 
\begin{align*}
T(z) = z \phi(T(z)),
\end{align*}
where the power series $\phi(t)=\sum_{j \geq 0} \phi_j t^j$ has only non-negative coefficients, 
$\phi_0>0$, and there is a $j\ge 2$ such that $\phi_j>0$. Moreover, it is required that the
equation $\tau\phi'(\tau)=\phi(\tau)$ has a unique positive solution. 

We are interested in the asymptotic protection number of a random simply generated tree, sampled
according to the weights from all simply generated trees with $n$ vertices, where $n$ tends to
infinity. 

\begin{rem}
For the sake of simplicity we assume that $\phi$ is non-periodic, meaning that there are integers
$i,j,k$ such that $\phi_i,\phi_j,\phi_k$ are all positive and satisfy
$\gcd(\phi_j-\phi_i,\phi_k-\phi_i)=1$. The periodic case can be dealt with in the very same way,
but the calculations leading to the desired number have to be done repeatedly (for analogous
situations) in order collect several contributions to the final value. 
\end{rem}

Within this paper the primary tool that is used will be singularity analysis (see
\cite{flajolet1990singularity,flajolet2009analytic}), which provides a direct connection between
the singularities of a generating function and the asymptotic behaviour of its coefficients.
By Pringsheim's theorem \cite[p.~240]{flajolet2009analytic} we know that a generating function
must have a singularity at $z=R$, if $R$ denotes the radius of convergence. Our assumption that
$\phi$ is non-periodic guarantees furthermore that this is the only singularity on the circle of
convergence. Throughout this paper we will call $z=R$ the \emph{dominant singularity} of the
generating function. In particular, we denote the dominant singularity of $T(z)$ by $\rho$.
Furthermore, we say that a function $f$ has an algebraic singularity of type $\alpha$ at $s$, if
there is a constant $C$ such that $f(z)\sim f(s)+ C\cdot(1-\frac{z}{s})^{\alpha}$ as $z$ tends
to $s$ in such a way that $z-s\notin \mathbb R^+$. In this case $f$ admits a Puiseux expansion
in terms of powers $(1-\frac{z}{s})^{\alpha/k}$ for some positive integer $k$. For instance, it
is well known that the generating function $T(z)$ associated to some class of simply generated
trees has an algebraic singularity $\rho$ of type 1/2 (for obvious reasons also called square root
singularity) the location of which is determined by the system $T(\rho)=\rho \phi(T(\rho))$,
$1=\rho \phi'(T(\rho))$, \emph{cf.} \cite{Drm09}. For further information on this theory we refer
the reader to \cite{flajolet2009analytic} and \cite{flajolet1990singularity}.

%

Let $T_k(z)$ denote the generating function of the class of simply generated trees that have protection number at least $k$, where $z$ marks the total number of nodes. Furthermore, let $\phi(T)$ be non-periodic. Then, $T_k(z)$ can be defined by
\begin{align}
\label{equ:gftk}
T_k(z)=z \( \phi(T_{k-1}(z)) - \phi_0 \).
\end{align}
Note that $T_0(z)=T(z)$.

\begin{lemma}
\label{lem:samesing}
All generating functions $T_k(z)$ have the same dominant singularity as $T(z)$, and it is a square root singularity.
\end{lemma}

\begin{proof}
First let us consider that the generating function $T_k(z)$ reads as
\begin{align*}
T_k(z)=
\Omega^k(T(z))
\end{align*}
where $\Omega(t)=z\phi(t)-z\phi_0$ and $\Omega^k(\cdot)$ denotes the $k$-fold composition. 
Since $\Omega(t)$ is analytic at $T(\rho)$, inserting a function
admitting a Puiseux expansion $t(z)=\alpha_0+\alpha_1\sqrt{1-\frac z\rho}+\dots$ results in 
\[
\Omega(t(z))=\Omega(\alpha_0)+\Omega'(\alpha_0) \alpha_1\sqrt{1-\frac z\rho} + \dots, 
\]
again being a Puiseux expansion at $z=\rho$. It is well known that $T(z)$ admits a Puiseux
expansion $\tau_0+\tau_1 \sqrt{1-\frac{z}{\rho}}+\dots $ with nonzero numbers $\tau_0$ and
$\tau_1$. Moreover, we always insert one of the functions $T_k(z)$, thus $\alpha_0$
attains the positive values $T_k(\rho)$, $k=0,1,2,\dots$, implying that $\Omega'(\alpha_0)$ is
always positive, as $\Omega(t)$ is a power series with only non-negative coefficients. By induction it
is guaranteed that $\alpha_1$ is always negative and thus all the function $T_k(z)$ have a unique
dominant singularity of square root type at $z=\rho$.
%
%
\end{proof}

In order to derive the expected value of the protection number $X_n$ of a random simply generated tree of size $n$ ($\textit{i.e.}$ with $n$ nodes) asymptotically, we use the well known formula
\begin{align}
\label{equ:expvalue}
\mathbb{E}X_n=\sum_{k \geq 1} \mathbb{P} ( X_n \geq k).
\end{align}

Thus, we need to calculate the probability $\mathbb{P} (X_n \geq k)$, which is given by
\begin{align*}
\mathbb{P} ( X_n \geq k) = \frac{[z^n]T_k(z)}{[z^n]T(z)}.
\end{align*}

\begin{thm}
Let $X_n$ be the protection number of a random simply generated tree of size $n$. Then the expected value $\mathbb{E}X_n$ and the variance $\mathbb{V}X_n$ 
satisfy
\begin{align*}
\lim_{n \to \infty} \mathbb{E}X_n = \sum_{k \geq 1} \rho^{k-1} \prod_{i=1}^{k-1} \phi'(T_i(\rho)),
\end{align*}
and 
\begin{align*}
\lim_{n \to \infty} \mathbb{V}X_n = \sum_{k \geq 1} (2k-1) \rho^{k-1} \prod_{i=1}^{k-1} \phi'(T_i(\rho))
-\(\lim_{n \to \infty} \mathbb{E}X_n\)^2.
\end{align*}
with $\rho$ denoting the dominant singularity of the generating function $T(z)=z\phi(T(z))$ of the class of simply generated trees.
\end{thm}

\begin{proof}
We know that the asymptotic behaviour of the generating function, namely $T(z)=\tau_0+\tau_1 \sqrt{1-\frac{z}{\rho}}+\tau_2 (1-\frac{z}{\rho})+\ldots$, implies
\begin{align}
\label{equ:simplygenzntz}
[z^n]T(z) \sim -\tau_1 \frac{n^{-3/2}}{\Gamma(-1/2)}\rho^{-n}, 
\end{align}
as $n$ tends to infinity. In order to derive the asymptotics of the $n$-th coefficient of $T_k(z)$, observe that 
we know from Lemma \ref{lem:samesing} that all generating functions $T_i(z)$ have the same dominant singularity $\rho$ of type $\frac{1}{2}$.
Setting $\eta=\sqrt{1-\frac{z}{\rho}}$, the Puiseux expansions of $T_k(z)$ and $T_{k-1}(z)$ read as
\begin{align*}
T_k(z)=\tau_{0,k}+\tau_{1,k} \eta + \tau_{2,k} \eta^2 + \ldots.
\end{align*}
and
\begin{align*}
T_{k-1}(z)=\tau_{0,k-1}+\tau_{1,k-1} \eta + \tau_{2,k-1} \eta^2 + \ldots.
\end{align*}

Plugging these expansions into (\ref{equ:gftk}) and using $z=\rho(1-\eta^2)$ we get
\begin{align*}
\tau_{0,k}+\tau_{1,k} \eta + \tau_{2,k} \eta^2 + \ldots = \rho(1-\eta^2) \( \sum_{j \geq 0} \phi_j \( \tau_{0,k-1}+\tau_{1,k-1} \eta + \tau_{2,k-1} \eta^2 + \ldots \)^j - \phi_0 \).
\end{align*}

Expanding and comparing coefficients of $\eta^0$ and $\eta^1$ yields
\begin{align*}
[\eta^0]&: \tau_{0,k}=\rho \phi(\tau_{0.k-1})-\rho\phi_0, \\
[\eta^1]&: \tau_{1,k}=\rho \sum_{j\geq0} \phi_j j \tau_{1,k-1} \tau_{0,k-1}^{j-1}.
\end{align*}

Obviously, the $\tau_{0,i}$'s match exactly the $T_i(\rho)$, $i\ge 0$, as they are the constant terms in the Puiseux expansions of the functions $T_i(z)$, with $0 \leq i \leq k$. Thus, the equation for $\tau_{1,k}$ can be rewritten as
$\tau_{1,k}= \rho \tau_{1,k-1} \phi'(T_{k-1}(\rho))$.

As $\tau_{1,0}=\tau_1$, we get
\begin{align*}
\tau_{1,k}= \tau_1 \rho^{k-1} \prod_{i=1}^{k-1} \phi'(T_i(\rho)).
\end{align*}

Applying a transfer lemma \cite{flajolet1990singularity} directly gives the asymptotics of the coefficients of $T_k(z)$ and plugging them in conjunction with 
\eqref{equ:simplygenzntz} into Equation \eqref{equ:expvalue} yields the asymptotic value for the mean.
In order to derive the formula for the asymptotic variance we use the equation 
\begin{align*}
\mathbb{V}X_n = \mathbb{E}(X_n^2) - (\mathbb{E}X_n)^2 \quad\text{ and }\quad \mathbb E(X_n^2) = 
\sum_{k \geq 1} (2k-1) \mathbb{P}(Y_n \geq k)  
\end{align*}
and immediately get the asserted result. 
\end{proof}

It is easy to see that the sequence $(T_i(\rho))_{i \geq 0}$ is monotonically decreasing, since the number of trees with protection number at least $i$ is always greater than the number of trees that have an $(i+1)$-protected root, \textit{i.e.} protection number at least $i+1$.
Since $\phi'$ is monotonically increasing on the positive real axis, this implies that
$
\rho \phi'(T_i(\rho)) \leq \rho \phi'(T_1(\rho)) < \rho \phi'(T(\rho)) = 1.
$
Thus, we can estimate the sum for the expected value by
\begin{align*}
\lim_{n \to \infty} \mathbb{E}X_n = \sum_{k \geq 1} \prod_{i=1}^{k-1} \( \rho \phi'(T_i(\rho)) \) < \sum_{k \geq 1} (\rho \phi'(T_1(\rho)))^{k-1},
\end{align*}
which converges, since $\rho \phi'(T_1(\rho))<1$. 
As the last sum is a convergent geometric series and the inequality even holds term-wise, 
we can calculate efficiently the asymptotic mean and variance
for all classes of simply generated trees with arbitrary accuracy. 
We will now exemplify this by calculating the limits of mean and variance of the protection number
of some prominent classes of simply generated trees.

\begin{ex}[Plane trees]
\label{ex:plane}
The generating function $C(z)$ of plane trees is the unique power series solution of 
\begin{align*}
C(z)=z \frac{1}{1-C(z)},
\end{align*}
which yields
\begin{align}\label{plane_gf}
C(z) = \frac{1}{2} - \sqrt{\frac{1}{4}-z}.
\end{align}
Thus, its dominant singularity is $\rho=\frac{1}{4}$, and $C(\rho)=\frac{1}{2}$.

The recursion for the $T_i(\rho)$'s reads as
\begin{align*}
T_1(\rho)=\frac{1}{4}, \qquad
T_i(\rho)= \frac{1}{4-4T_{i-1}(\rho)} - \frac{1}{4}.
\end{align*}

\noindent
In case of plane trees the recursion can be solved explicitly, leading to 
\begin{align*}
T_i(\rho) = \frac{3}{2 (4^i+2)}.
\end{align*}

\noindent
The limits of expected value and variance are therefore given by
\begin{align*}
\lim_{n\to\infty}\mathbb{E}X_n = \sum_{k \geq 1} \frac{1}{4^{k-1}} \prod_{i=1}^{k-1} \frac{1}{\( 1
- \frac{3}{2(4^i+2)} \)^2} \approx 1.622971384715353,
\end{align*}
and
\begin{align*}
\lim_{n\to\infty}\mathbb{V}X_n = \sum_{k \geq 1} (2k-1) \frac{1}{4^{k-1}} \prod_{i=1}^{k-1}
\frac{1}{\( 1 - \frac{3}{2(4^i+2)} \)^2} 
-\(\lim_{n \to \infty} \mathbb{E}X_n\)^2 
\approx 0.7156950717833327,
\end{align*}
which has already been calculated by Heuberger and Prodinger in \cite{heuberger2017protection}.
\end{ex}

\begin{ex}[Motzkin trees]
\label{ex:motzkin}
The generating function $M(z)$ of Motzkin trees is defined by
\begin{align*}
M(z)= z \(1+M(z)+M(z)^2 \), 
\end{align*}
which can be solved to result in
\begin{align*}
M(z)=\frac{1-z-\sqrt{1-2z-3z^2}}{2z}.
\end{align*}
Thus, its dominant singularity is $\rho=\frac{1}{3}$ and $M(\rho)=1$.

The recursion for the $T_i(\rho)$'s reads as
\begin{align*}
T_1(\rho)=\frac{2}{3},    \qquad
T_i(\rho)= \frac{1}{3}\(T_{i-1}(\rho)^2 + T_{i-1}(\rho)\)
\end{align*}

This recursion can be transformed into another one for the numerators of the rational numbers
$T_i(\rho)$: Indeed, if we write $T_i(\rho)=A_i \cdot 3^{-2^i+1}$, then $A_1=2$ and
$A_i=A_{i-1}^2+3^{2^{i-1}-1}\cdot A_{i-1}$, for $i\ge 2$.  The recurrence for the $A_i$'s does
not fall into the scheme of Aho and Sloane~\cite{AS73} and we are not aware of any method to
solve it explicitly. But as stated before, the sequence $(T_i(\rho))_{i\ge 1}$ is exponentially
decreasing and estimates are easily obtained. Thus we can calculate the limits of mean and 
variance for the protection number numerically with arbitrary accuracy:
\begin{align*} 
\lim_{n \to \infty} \mathbb{E}X_n \approx 2.546378248338912,\qquad 
\lim_{n \to \infty} \mathbb{V}X_n \approx 1.679348871220563.
\end{align*} 
\end{ex}

\begin{ex}[Incomplete binary trees]
\label{ex:increasingbinary}
The generating function $I(z)$ of incomplete binary trees is defined by
\begin{align*}
I(z)= z \(1+2I(z)+I(z)^2 \), 
\end{align*}
which gives
\begin{align*}
I(z)=\frac{1-2z-\sqrt{1-4z}}{2z}.
\end{align*}
The dominant singularity is therefore at $\rho=\frac{1}{4}$ and $I(\rho)=1$.

The recursion for the $T_i(\rho)$'s reads as
\begin{align*}
T_1(\rho)=\frac{3}{4}, \qquad
T_i(\rho)= \frac{1}{4}(T_{i-1}(\rho)^2+ 2 T_{i-1}(\rho)).
\end{align*}

This recursion cannot be solved explicitly, but the numerical values can be easily computed: They 
are 
\begin{align*}
\lim_{n \to \infty} \mathbb{E}X_n \approx 3.536472483525321,\qquad
\lim_{n \to \infty} \mathbb{V}X_n \approx 3.763883442795153.
\end{align*}
\end{ex}

\begin{ex}[Cayley trees]
\label{ex:cayley}
Though, in a strict sense, Cayley trees do not belong to the class of simply generated trees (\emph{cf.} the discussions in \cite{Ja12} and \cite{GJW}), 
they are usually listed as an example for that class. In fact, they are closely related (see \cite{PaSt18} for a thorough analysis and \cite{GJW} for an analysis of the differences) and in many contexts (like the one considered here), quotients of coefficients are computed which makes the fact that in this case the generating functions are exponential ones irrelevant.
 
The (exponential) generating function $C(z)$ of Cayley trees is defined by
\begin{align*}
C(z)= z e^{C(z)}, 
\end{align*}
which has its dominant singularity at $\rho=\frac{1}{e}$. Moreover, we have $C(\rho)=1$.

The recursion for the $T_i(\rho)$'s reads as
\begin{align*}
T_1(\rho)=1-\frac{1}{e}, \qquad
T_i(\rho)= \frac{1}{e}(e^{T_{i-1}(\rho)}-1).
\end{align*}

As in the two previous examples the recursion for the $T_i(\rho)$'s cannot be solved explicitly, but the numerical values are
\begin{align*}
\lim_{n \to \infty} \mathbb{E}X_n \approx 2.286198316708012,\qquad
\lim_{n \to \infty} \mathbb{V}X_n \approx 1.598472890455086.
\end{align*}
\end{ex}

\begin{ex}[Binary trees]
This is the class of complete binary trees with only internal vertices contributing to the size. The generating function
is then defined by the functional equation $B(z)=1+zB(z)^2$ with $B(z)=C(z)/z$ where $C(z)$ is the
function displayed in \eqref{plane_gf}. Though this class does not strictly
fall into the simply generated framework, the functional equation is of the form
$B(z)-1=z\phi(B(z)-1)$, which reflects the fact that incomplete binary trees with all nodes
counted are in bijection to complete binary trees with only internal vertices counted. For the
protection number this causes some shifts within the tree. But the methodology presented above 
works here as well. We get $T_0(z)=B(z)$ and $T_k(z)=zT_{k-1}(z)^2.$ Since $\rho=1/4$ we have
$T_k(\rho)= 2^{2-2^k}$, for all $k\ge 0$, and then finally $\mathbb{P}(X_n\ge k)\to 2^{k+1-2^k}$, as $n$ tends to
infinity. Thus we obtain 
\begin{align*}
\lim_{n \to \infty} \mathbb{E}X_n 
\approx 1.562988296151161, \qquad
\lim_{n \to \infty} \mathbb{V}X_n 
\approx 0.372985688954940.
\end{align*} 
\end{ex}

\subsection{Protection number of a random vertex}
\label{SGT_random_vertex}

In the first part of this section we studied the average protection number of a simply generated tree, that is the protection number of the root of the simply generated tree. 
Now we are interested in the average protection number of a randomly chosen vertex in a simply generated tree of size $n$. We denote this sequence of random variables by $Y_n$.

As in the previous section we calculate the mean via
$
\mathbb{E}Y_n = \sum_{k \geq 1} \mathbb{P}(Y_n \geq k).
$
%
In order to do so we proceed analogously to Heuberger and Prodinger in
\cite{heuberger2017protection} and define $S_k(z)$ to be the generating function of the sequence
$(s_{n,k})_{n \geq 0}$ of $k$-protected vertices summed over all trees of size $n$.
As in \cite{heuberger2017protection} this generating function can be calculated by
\begin{align}\label{HP_formula}
S_k(z)=z^{-1}T_k(z) \frac{\partial}{\partial u} T(z,1),
\end{align} 
by means of the bivariate generating function $T(z,u)$ of simply generated trees, where $z$ marks the size and $u$ the number of leaves, and the generating function $T_k(z)$ of simply generated trees with protection number at least $k$.
The formula for $S_k(z)$ arises from considering a $k$-protected vertex in the following way: First point at a leaf in a simply generated tree (which yields the factor $\frac{\partial}{\partial u} T(z,1)$), then remove this leaf (which explains the $z^{-1}$) and finally attach a tree with protection number at least $k$ (giving the factor $T_k(z)$).

\brem
The procedure works also for complete binary trees, where only internal vertices contribute to
the tree size. The only difference is that for complete binary trees the factor $z^{-1}$ in
\eqref{HP_formula} must be removed, because removing a leaf does not change the size. 
\erem

Using the generating function $S_k(z)$ we can express the probability $\mathbb{P}(Y_n \geq k)$ by
\begin{align}
\label{equ:propyngeqk}
\mathbb{P}(Y_n \geq k) = \frac{[z^n] S_k(z)}{n[z^n]T(z)}.
\end{align}  

\begin{thm}
Let $Y_n$ be the protection number of a randomly chosen vertex in a random simply generated tree of size $n$. Then,
\begin{align*}
\lim_{n \to \infty} \mathbb{E}Y_n = \frac{\phi_0}{T(\rho)} \sum_{k \geq 1} T_k(\rho),
\end{align*}
and
\begin{align*}
\lim_{n \to \infty} \mathbb{V}Y_n = \frac{\phi_0}{T(\rho)} \sum_{k \geq 1}(2k-1) T_k(\rho) 
-\(\lim_{n \to \infty} \mathbb{E}Y_n\)^2.
\end{align*}
\end{thm}

\begin{proof}
First we need to determine the $n$-th coefficient of $S_k(z)$. We have
\begin{align}
\label{equ:partielleableitungtz1}
\frac{\partial}{\partial u} T(z,1) = \frac{z \phi_0}{1 - z  \phi'(T(z))}.
\end{align}

Using $T'(z) = z \phi'(T(z)) T'(z) + \phi(T(z))$ and $\phi(T(z))=\frac{T(z)}{z}$ we get
\begin{align*}
z \phi'(T(z)) = \frac{T'(z) - \frac{T(z)}{z}}{T'(z)}.
\end{align*}

\noindent
Therefore (\ref{equ:partielleableitungtz1}) transforms to
\begin{align*}
\frac{\partial}{\partial u} T(z,1) = \frac{T'(z) z^2 \phi_0}{T(z)}.
\end{align*}

\noindent
Thus, altogether we have
\begin{align*}
[z^n]S_k(z) = [z^n] z^{-1} T_k(z) \frac{T'(z) z^2 \phi_0}{T(z)},
\end{align*}
which gives
\begin{align*}
[z^n]S_k(z) \sim \frac{- \tau_{0,k} \tau_1 \phi_0}{2 \tau_0} \frac{n^{-1/2}}{\Gamma(1/2)} \rho^{-n}.
\end{align*}

\noindent
Finally, we get
\begin{align*}
\mathbb{E}Y_n &= \sum_{k \geq 1} \mathbb{P}(Y_n \geq k) = \sum_{k \geq 1} \frac{[z^n] S_k(z)}{n[z^n]T(z)} 
\overset{n \to \infty}{\to} \sum_{k \geq 1} \frac{T_k(\rho) \phi_0}{T(\rho)}.
\end{align*}

\noindent
For the variance we use again the formula 
$\mathbb{V}Y_n = 
\sum_{k \geq 1} (2k-1) \mathbb{P}(Y_n \geq k)
-\mathbb E(Y_n)^2$
and (\ref{equ:propyngeqk}).
\end{proof}

\begin{table}[h!]
\begin{tabular}{l|c|c}
&  $\lim_{n \to \infty} \mathbb{E}Y_n$ & $\lim_{n \to \infty} \mathbb{V}Y_n $ \\ \hline
Plane trees & 0.7276492769137261 & 0.8168993794836289 \\
Motzkin trees & 1.307604625963334 & 1.730614214799486 \\
Incomplete binary trees & 1.991819588602741 & 3.638259051495130 \\
Cayley trees & 1.186522661652180 & 1.632206223956926 \\
Complete binary trees & 1.265686036087572 & 0.226591112528581
\end{tabular}
\vspace{3mm}
\caption{The approximate values for the limits of mean and variance of the protection number of a random vertex in different classes of simply generated trees.}
\end{table}

\section{P\'olya trees}\label{PT}

\subsection{Protection number of the root}
Let $T(z)$ be the generating function of P\'olya trees, which reads as
\begin{align*}
T(z) = z e^{T(z)} \exp \left( \sum_{i \geq 2} \frac{T(z^i)}{i} \right),
\end{align*}
and in correspondence to the previous section let us denote by $T_k(z)$ the generating function of the class of P\'olya trees that have protection number at least $k$. This generating function can be specified by
\begin{align}
\label{equ:polyatk}
T_k(z)=z e^{T_{k-1}(z)} \exp \left( \sum_{i \geq 2} \frac{T_{k-1}(z^i)}{i} \right) - z,
\end{align}
with $T_0(z)=T(z)$. From the classical results of P\'olya~\cite{Polya1937} we know that $T(z)$ has
a unique dominant singularity $\rho$ of type 1/2 and admits Puiseux series expansion there, which 
starts as 
\begin{equation} \label{Polya_Puiseux}
T(z)\sim 1-b\sqrt{1-\frac z\rho}+\frac{b^2}3\(1-\frac z\rho\)+d\(1-\frac z\rho\)^{3/2}+\cdots.
\end{equation} 
Numerical approximations for the constants have been
first computed by Otter~\cite{otter1948}. This was also topic in Finch~\cite[Section~5.6]{Finch03}
and \cite[p.~477]{flajolet2009analytic} where we find approximations up to 25 digits: $\rho\approx
0.3383218568992076951961126$ and $b\approx 1.55949002037464088554226$.

\begin{lemma}
All the generating functions $T_k(z)$ have their (unique) dominant singularity at $\rho$, and the
singularity is a square root singularity.
\end{lemma}

\begin{proof}
First let us recall that $T_0(z)=T(z)$. Thus, for $k=0$ the lemma is trivial. For $k \geq 1$ we
proceed by induction. Therefore let us assume that $T_{k-1}(z)$ has the dominant singularity
$\rho$ which is of type $\frac{1}{2}$. Then the dominant singularity of $T_k(z)$, satifying the
recurrence relation \eqref{equ:polyatk}, 
comes from $e^{T_{k-1}(z)}$, since $\exp \left( \sum_{i \geq 2} \frac{T_{k-1}(z^i)}{i} \right)$ is analytic in $|z| < \rho + \epsilon$ with $\epsilon > 0$ sufficiently small. 
Applying the exponential function to a function having an algebraic singularity does neither change the location nor the type of the singularity, which proves the assertion after all.
\end{proof}

The goal of this section is to derive an asymptotic value for the average protection number of P\'olya trees. We use again the formula 
$\mathbb{E}X_n = \sum_{k \geq 1} \mathbb{P} (X_n \geq k)$, 
%
but rewrite this equation as
\begin{align*}
\mathbb{E}X_n = \sum_{k \geq 1} \prod_{i=1}^k \mathbb{P} (X_n \geq k | X_n \geq k-1),
\end{align*}
where the conditional probabilities can be obtained by
\begin{align}
\label{equ:condprobpolya}
\mathbb{P} (X_n \geq k | X_n \geq k-1) =  \frac{[z^n]T_{k}(z)}{[z^n]T_{k-1}(z)}.
\end{align}

\begin{lemma}
\label{lem:expansionstktk+1}
The asymptotic expansions of the $n$-th coefficients of $T_{k}(z)$ and $T_{k-1}(z)$ read as
\begin{align*}
[z^n]T_{k-1}(z) &= \frac{\gamma_k \rho^{-n} n^{-\frac{3}{2}}}{\Gamma(-1/2)} \left( 1 + \mathcal{O} \left( \frac{1}{n} \right) \right), \\
[z^n]T_{k}(z) &= \frac{(T_{k}(\rho)+\rho)\gamma_k \rho^{-n} n^{-\frac{3}{2}}}{\Gamma(-1/2)} \left( 1 + \mathcal{O} \left( \frac{1}{n} \right) \right),
\end{align*}
as $n \to \infty$, with a constant $\gamma_k>0$.
\end{lemma}

\begin{proof}
Let the Puiseux expansion of $T_{k-1}(z)$ be given by $T_{k-1}(z) = T_{k-1}(\rho) - \gamma_k \sqrt{1- \frac{z}{\rho}} + \ldots$.

Then $T_{k}(z)$ behaves asymptotically as
$
T_{k}(z) \sim \rho e^{T_{k-1}(\rho)} Q_{k-1}(\rho) e^{-\gamma_k \sqrt{1-\frac{z}{\rho}}},
$
where $Q_{k-1}(\rho) = \exp \left( \sum_{i \geq 2} \frac{T_{k-1}(\rho^i)}{i} \right)$.
Applying the asymptotic relation $e^{-\gamma_k \sqrt{1-\frac{z}{\rho}}} \sim 1 -\gamma_k \sqrt{1-\frac{z}{\rho}}$ and using the equation $\rho e^{T_{k-1}(\rho)} Q_{k-1}(\rho) = T_{k}(\rho) + \rho$ completes the proof.
\end{proof}

Plugging the expansions obtained in Lemma \ref{lem:expansionstktk+1} into Equation (\ref{equ:condprobpolya}) gives
\begin{align*}
\mathbb{P} (X_n \geq k | X_n \geq k-1) = T_{k}(\rho)+\rho,
\end{align*}
which directly yields the following theorem.

\begin{thm}
\label{thm:polyaexxn}
Let $X_n$ be the protection number of a random P\'olya tree of size $n$. Then
\begin{align}
\label{equ:limexnpolya}
\lim_{n \to \infty} \mathbb{E}X_n = \sum_{k \geq 1} \prod_{i=1}^k ( T_{k}(\rho) + \rho) \approx 2.154889671973873,
\end{align}
and
$\lim_{n \to\infty} \mathbb{V}X_n \approx 
1.369993017502652.$
\end{thm}

\begin{proof}
The proof for the asymptotic mean follows directly by Lemma \ref{lem:expansionstktk+1}.
In order to determine the variance we use the representation $\lim_{n\to\infty}\mathbb{V}X_n =
\sum_{k \geq 1} (2k-1) \prod_{i=1}^k (T_k(\rho)+\rho)-\mathbb E(X_n)^2$.
\end{proof}

\begin{rem} 
Note that in order to get accurate numerical values, we must not compute $T_k(\rho)$ by insertion
into a (truncated) series expansion for $T_k(z)$ The reason is that $\rho$ lies on the circle of
convergence and thus the convergence is very slow at $z=\rho$. Instead, $T_k(\rho)$ can be
directly computed using the recurrence relation \eqref{equ:polyatk}. The values $T_k(\rho^i)$
for $i\ge 2$, which appear in that recurrence relation, can be computed with the help of the
series expansion of $T_k(z)$, because $\rho^i$ then lies in the interior of region of convergence
where the series converges at an exponential rate.  
\end{rem}

\begin{rem}
We could also have used the same approach as for simply generated trees in order to get the asymptotic mean. Then the resulting formula looks like
\begin{align}
\label{equ:limexnpolyaaltemethode}
\lim_{n \to \infty} \mathbb{E}X_n = \sum_{k \geq 1} \rho^{k-1} \prod_{i=1}^{k-1} C_i e^{T_i(\rho)},
\end{align} 
where $C_j = \exp \left( \sum_{i \geq 2} \frac{T_j(\rho^i)}{i} \right)$.
One can show that $C_i$ tends to 1 and and $T_i(\rho)$ tends to 0 exponentially fast and get the constant given in Theorem \ref{thm:polyaexxn}. However, since this approach requires more technical calculations, we decided to switch to the more direct strategy using the conditional probabilities. Moreover note that the equivalence of (\ref{equ:limexnpolya}) and (\ref{equ:limexnpolyaaltemethode}) is immediate from (\ref{equ:polyatk}).
\end{rem}

\subsection{Protection number of a random vertex}

The method of marking a leaf and replacing it by a tree with protection number $k$ does not work
here. Due to possible symmetries in non-plane trees, this would result in wrong counting: Indeed,
if there are $k$-protected vertices $x_1,\dots,x_\ell$ which can be mapped to each other by some
automorphisms of the tree (\emph{i.e.}, they lie in the same vertex class), then only one of them
is counted. Though this is counterbalanced by trees having $\ell$ leaves in the same vertex class one of
which is replaced by a tree with protection number $k$ (the root of this tree is then counted
$\ell$ times), there are further overcounts: As all leaves are marked, trees having several leaves
in the same vertex class are counted several times, and so are their $k$-protected vertices. 

Thus we appeal to the proof of \cite[Theorem~3.1]{Wa15} here: For a tree $T$ let 
\[ 
f(T)=
\begin{cases}
1 & \text{ if $T$ has protection number at least $k$}, \\ 
0 & \text{ otherwise.}
\end{cases}
\]
Moreover, we define $F(T)$ to be the number of $k$-protected nodes in $T$. Then the generating
function $R_k(z,u)=\sum_T z^{|T|} u^{F(T)}$ satisfies (\emph{cf.} \cite[Equ.~(3.1)]{Wa15})
\begin{equation} \label{Wagner_formula}
z\exp\(\sum_{i\ge 1}\frac{R_k(z^i,u^i)}i\)=\sum_{n\ge 1} z^n \sum_{T:|T|=n} u^{F(T)-f(T)} 
\end{equation} 

As in Section~\ref{SGT_random_vertex} we utilize the formula 
$\mathbb{E}Y_n = \sum_{k \geq 1} \mathbb{P}(Y_n \geq k)$ and express the
occurring probabilities as $\mathbb{P}(Y_n \geq k)=[z^n]S_k(z)/(n[z^n]T(z))$ with $S_k(z)$ being
the generating function whose $n$th coefficient is the cumulative number of $k$-protected nodes in
all trees of size $n$. Obviously, $( (\partial/\partial u) R_k) (z,1)=S_k(z)$ and thus 
by differentiating \eqref{Wagner_formula} with respect to $u$ and inserting $u=1$ we obtain 
\begin{equation} \label{funeq_Sk}
T(z)\sum_{i\ge 1} S_k(z^i) = S_k(z)-T_k(z).  
\end{equation} 
This implies  
\begin{equation} \label{S_k_Polya}
S_k(z)=\frac{T(z)\sum_{i\ge 2} S_k(z^i) + T_k(z)}{1-T(z)} \sim \frac{\sum_{i\ge 2} S_k(\rho^i) +
T_k(\rho)}{b \sqrt{1-\frac z\rho}}
\end{equation} 
where $b$ is the constant appearing in \eqref{Polya_Puiseux}. Standard transfer theorems applied
to \eqref{Polya_Puiseux} give 
\[
[z^n]T(z) \sim \frac{-b n^{-3/2} \rho^{-n}}{\Gamma(-1/2)} = \frac{b n^{-3/2}
\rho^{-n}}{2\sqrt{\pi}}, 
\]
and from \eqref{S_k_Polya} we get 
\[
[z^n]S_k(z)\sim \frac{\(\sum_{i\ge 2} S_k(\rho^i) + T_k(\rho)\) n^{-1/2} \rho^{-n}}{b\sqrt{\pi}}
\]
and thus 
\begin{equation} \label{above_expression}
\mathbb{P}(Y_n \geq k) \sim \frac{2}{b^2}\(\sum_{i\ge 2} S_k(\rho^i) + T_k(\rho)\).  
\end{equation} 

Since $T_k(\rho)$ decreases exponentially (\emph{cf.} remark after Theorem~\ref{thm:polyaexxn}),
and so does $\sum_{i\ge 2} S_k(\rho^i)$, these probabilities decrease 
exponentially and thus the series for $\mathbb{E}Y_n$, namely
\[
\mathbb{E}Y_n=\sum_{k\ge 1} \mathbb P(Y_n\ge k), 
\]
converges rapidly. But \eqref{above_expression}
still bears a secret, because we do not have an explicit expression for $S_k(z)$ and we cannot
solve the functional equation \eqref{funeq_Sk}. 

For numerical purposes, however, it is not necessary to have an explicit expression for $S_k(z)$.
If we write $S_k(z)=\Psi(S_k(z))$ with $\Psi$ being the operator on the ring of formal power
series defined by 
\[
\Psi(f(z))=\frac{T(z)\sum_{i\ge 2} f(z^i) + T_k(z)}{1-T(z)}, 
\]
then $\Psi$ is a contraction on the metric space $\mathbb R[[z]]$ equipped with the formal
topology (\emph{cf.} \cite[Appendix~A.5]{flajolet2009analytic}). Indeed, if $f(z)$ and $g(z)$
coincide up to their $\ell$th coefficient, then the first $2\ell+2$ coefficients of $\Psi(f(z))$
and$\Psi(g(z))$ coincide. 

As there is exactly one tree with $k+1$ vertices which possesses $k$-protected vertices at all
(namely the path of length $k$ has a $k$-protected root) whereas all smaller trees do not possess
any $k$-protected vertices, we know that the (one-term) series $z^{k+1}$ coincides with
$S_k(z)=z^{k+1}+\cdots$ in its first $k+2$ coefficients. Applying $\Psi$ to $z^{k+1}$ a few times,
with each application more than doubling the number of known coefficients of $S_k(z)$, gives
quickly a fairly accurate expression for $S_k(z)$. We obtain the following theorem:

\begin{thm}
Let $Y_n$ be the protection number of a random vertex in a random P\'olya tree of size $n$. Then
\begin{align*}
\lim_{n \to \infty} \mathbb{E}Y_n = \sum_{k \geq 1} \frac{2}{b^2}\(\sum_{i\ge 2} S_k(\rho^i) +
T_k(\rho)\) \approx 0.9953254987
\end{align*}
and
$\lim_{n \to \infty} \mathbb{V}Y_n \approx 
1.3818769746.$
\end{thm}

\section{Non-plane binary trees}\label{NPBT}

\subsection{Protection number of the root}

We denote by $T(z)$ the generating function of non-plane binary trees, where $z$ marks the number of internal nodes. Then $T(z)$ satisfies
\begin{align*}
T(z)=1+z \left( \frac{1}{2} T(z)^2 + \frac{1}{2} T(z^2) \right).
\end{align*}

The generating function $T_k(z)$ of non-plane binary trees with protection number at least $k$ fulfills
\begin{align*}
T_k(z)=z \left( \frac{1}{2} T_{k-1}(z)^2 + \frac{1}{2} T_{k-1}(z^2) \right),
\end{align*}
and $T_0(z)=T(z)$.

In order to obtain the asymptotic mean and variance for the protection number of a random non-plane binary tree of size $n$ we proceed analogously as in the previous section for P\'olya trees. Thus, we use
\begin{align*}
\mathbb{E}X_n = \sum_{k \geq 1} \prod_{i=1}^k \mathbb{P}(X_n \geq k | X_n \geq k-1) = \sum_{k \geq 1} \prod_{i=1}^k \frac{[z^n]T_k(z)}{[z^n]T_{k-1}(z)}.
\end{align*}

\begin{thm}
Let $X_n$ be the protection number of a random non-plane binary tree of size $n$. Then
\begin{align*}
\lim_{n \to \infty} \mathbb{E}X_n = \sum_{k \geq 1} \prod_{i=1}^{k-1} (\rho T_i(\rho)) 
\approx 1.707603060723366
\end{align*}
and
$\lim_{n \to \infty} \mathbb{V}X_n 
\approx 
0.431102549825064.$
\end{thm}

\begin{proof}
Let the Puiseux expansion of $T_k(z)$ and $T_{k+1}(z)$ read as
\begin{align*}
T_{k-1}(z) = T_{k-1}(\rho) - \gamma_k \sqrt{1-\frac{z}{\rho}} + \mathcal{O} \left( 1- \frac{z}{\rho} \right),
\end{align*}
and
\begin{align*}
T_{k}(z) = \rho \left( \frac{1}{2} T_{k-1}(\rho)^2 +\frac{1}{2} T_{k-1}(\rho^2) \right) + \rho T_{k-1}(\rho) \gamma_k \sqrt{1-\frac{z}{\rho}} + \mathcal{O} \left( 1- \frac{z}{\rho} \right)
\end{align*}

Using singularity analysis yields the desired result for the mean. For the variance we use again
the formula $\mathbb{V}X_n = \sum_{k\ge 1} (2k-1)\mathbb P(X_n\ge k)-\mathbb E(X_n)^2$.
\end{proof}

\subsection{Protection number of a random internal vertex}

The asymptotic mean and variance for the protection number of a randomly chosen internal vertex in a random non-plane binary tree can be obtained in the same way as in the previous section for P\'olya trees.

Thus, we again set up an equation for the generating function $R_k(z,u)$ where the coefficients $[z^nu^l]R_k(z,u)$ count the number of non-plane binary trees of size $n$ with $l$ $k$-protected vertices:
\begin{align*}
\frac{z}{2} \( R_k(z,u)^2+R_k(z^2,u^2) \) = \sum_{n \geq 1} z^n \sum_{T:|T|=n} u^{F(T)-f(T)}
\end{align*}
Differentiating this equation with respect to $u$ and setting $u=1$ yields
\begin{align*}
zT(z)S_k(z)+zS_k(z^2)=S_k(z)-T_k(z).
\end{align*}
Therefore we get
\begin{align*}
S_k(z)=\frac{zS_k(z^2)+T_k(z)}{1-zT(z)}.
\end{align*}
The asymptotic expansion of $T(z)$ is given by
\begin{align*}
T(z) \sim \frac{1}{\rho} - a\sqrt{1-\frac{z}{\rho}}.
\end{align*}
In \cite[p.~477]{flajolet2009analytic} we find the numerical values of the constants $\rho$ and
$a$. (\emph{Caveat}: The scaling is different, so \cite[p.~477]{flajolet2009analytic} in fact
lists $a\cdot\rho$, not $a$.) We have
$\rho\approx 0.4026975036714412909690453$ and $a\approx 2.8061602222420538943722824$.
Using this expansion we get 
\begin{align*}
\mathbb{P}(Y_n \geq k) = \frac{[z^n]S_k(z)}{n[z^n]T(z)} \sim \frac{2}{a^2\rho} \( \rho S_k(\rho^2)+T_k(\rho) \).
\end{align*}

By denoting $\Psi(f(z))=\frac{zf(z^2)+T_k(z)}{1-zT(z)}$ we can use the same arguments as in the
P\'olya case to efficiently obtain numerical values for the probabilities $\mathbb{P}(Y_n \geq k)$. Finally, we are able to calculate the asymptotic mean and variance for the protection number of a random node in non-plane binary trees.

\begin{thm}
Let $Y_n$ be the protection number of a random internal vertex in a random non-plane binary tree of size $n$. Then
\begin{align*}
\lim_{n \to \infty} \mathbb{E}Y_n &= \frac{2}{a^2\rho} 
\sum_{k \geq 1} \( \rho S_k(\rho^2)+T_k(\rho) \)
\approx 1.3124128299,
\end{align*}
and $\lim_{n \to \infty} \mathbb{V}Y_n \approx 
0.2676338724.$
\end{thm}

\section{Conclusion}

In this paper we generalized the work of Heuberger and Prodinger, who derived the average
protection number of plane trees, to a more general framework. We obtained the average protection
number for all simply generated trees, as well as for P\'olya trees and non-plane binary trees. We
did not include P\'olya trees with general degree restrictions, since the general expressions will
look clumsy and only numerical results for specific classes may be of interest. But it is
immediate that the asymptotic mean and variance of the protection number for P\'olya-trees with
any kind of degree restriction can be calculated in the very same way. As we saw in some of the
examples, there are classes of trees, for which the obtained formulas involve a recurrence that
might not be solvable explicitly. However, using these equations it is possible to calculate the
asymptotic mean and variance in an arbitrarily accurate way with a fairly low computational
effort. In Table~\ref{tab:meanprotnr} we summarize the obtained results for some specific tree
classes.

\begin{table}[h!]
\begin{tabular}{l|c|c}
Tree model & $\lim_{n \to \infty} \mathbb{E}X_n$ & $\lim_{n \to \infty} \mathbb{E}Y_n$ \\
\hline
\textbf{Simply generated trees} &&\\
Plane trees & 1.62297 & 0.72765 \\
Motzkin trees & 2.54638 & 1.30760 \\
Incomplete binary trees & 3.53647 & 1.99182 \\
Cayley trees & 2.28620 & 1.18652 \\
Complete binary trees & 1.56298 & 1.26568 \\
\textbf{Non-plane trees} & & \\
P\'olya trees & 2.15489 & 0.99532 \\
Non-plane binary trees & 1.70760 & 1.31241
\end{tabular}
\vspace{3mm}
\caption{Summary of the obtained mean values for the protection numbers.}
\label{tab:meanprotnr}
\end{table}

It is well known that Cayley trees and P\'olya trees are very similar, but the latter are not
simply generated, as the simple proof presented in \cite{DrGi10} shows. A detailed analysis of the
structural differences was done in \cite{GJW,PaSt18}: Roughly speaking, P\'olya trees are Cayley
tree (more precisely, the simply generated class whose ordinary generating function is the
exponential generating function of Cayley trees) with small forests attached to each vertex. 
Comparing the resulting values (from Table \ref{tab:meanprotnr}) for Cayley trees and P\'olya
trees shows the quantitative effect of those forests, which have on average less than one vertex.
As expected, these additional forests decrease the protection numbers.

For complete binary trees the correspondence between plane and non-plane is a bit different due to
the strict degree constraint. The small forests are not attached anywhere, but they always
consist of two identical trees and attachment is done by replacing a leaf. The effect of the
presence or absence of symmetries seems stronger than the possible increase of the protection
number by adding forests, because the larger number of plane structures gives some bias to lower
protection numbers.

\bibliographystyle{plain}
\bibliography{references}

\end{document}